\documentclass[11pt]{amsart}

\usepackage{amsmath,amssymb,amscd}

\usepackage{amsrefs}
\usepackage{enumerate}
\usepackage{xcolor}

\usepackage{amsmath,amssymb}

\usepackage{amsxtra, amsmath}
\usepackage{amssymb, amscd}
\usepackage{graphicx}
\usepackage{mathrsfs}

\usepackage{mathtools}
\usepackage{float}

\newtheorem{thm}{Theorem}[section]
\newtheorem{prop}[thm]{Proposition}

\newtheorem{cor}[thm]{Corollary}

\theoremstyle{definition}
\newtheorem{definition}[thm]{Definition}

\theoremstyle{remark}
\newtheorem{remark}[thm]{Remark}

\usepackage{dsfont}
\allowdisplaybreaks

\title{Matrix-valued bispectral discrete orthogonal polynomials}
\author{Ignacio Bono Parisi}

\subjclass[2020]{33C45, 42C05, 39A70, 34L10}
\thanks{This paper was partially supported by SeCyT-UNC, PIP 33620230100819CB, CONICET, PIP 1220150100356.}
\keywords{Discrete orthogonal polynomials, Matrix-valued orthogonal polynomials, Difference operators, Discrete-discrete bispectrality, Matrix-valued bispectral functions}

\address{CIEM-FaMAF\\ Universidad Nacional de C\'or\-do\-ba\\
CP 5000, C\'or\-do\-ba,  Argentina}
\email{ignacio.bono@unc.edu.ar}

\begin{document}
\begin{abstract}
    We develop a unified construction of matrix-valued orthogonal polynomials associated with discrete weights, yielding bispectral sequences as eigenfunctions of second-order difference operators. This general framework extends the discrete families in the classical Askey scheme to the matrix setting by producing explicit matrix analogues of the Krawtchouk, Hahn, Meixner, and Charlier polynomials. Our results include explicit expressions for the weights, the orthogonal polynomials, and the corresponding difference operators. 
\end{abstract} 
\maketitle

\section{Introduction}
The classical families of discrete orthogonal polynomials in the Askey scheme —such as Charlier, Meixner, Hahn, and Krawtchouk polynomials— are characterized by their bispectral nature: they satisfy a three-term recurrence relation and are eigenfunctions of a second-order difference operator. In fact, these four families are the only scalar orthogonal polynomials supported on the real line that satisfy such a difference equation with coefficients independent of the degree. These families play a fundamental role in analysis and combinatorics and find applications in areas such as numerical analysis, statistics, and coding theory.

In 1949, M. G. Krein introduced the matrix-valued extension of the theory of orthogonal polynomials \cites{K49,K71}. Decades later, in 1997, A. J. Durán formulated the matrix-valued version of Bochner’s problem \cite{D97}, which sparked significant interest in the continuous case. Since then, there has been remarkable progress, with a growing list of explicit examples of matrix-valued orthogonal polynomials that are eigenfunctions of second-order differential operators, see \cites{GPT05, CG06, DG04, G03, GPT01, GPT02, DG07, PT07, P08, PR08, PZ16, KPR12, BP24-1}. This progress also includes a classification of the matrix Bochner problem under additional hypotheses \cite{CY18}. These developments have revealed profound connections with representation theory, approximation theory, harmonic analysis, operator theory, and the theory of special functions, showcasing the conceptual depth and algebraic richness of matrix-valued orthogonal polynomials.

In contrast to the continuous case, the discrete matrix-valued theory has received comparatively less attention. Some relevant contributions in this direction include \cites{D12,ADR13,DS14,EMR24}. Explicit expressions for the orthogonal polynomials are rarely available, and known examples are limited to certain Charlier and Meixner-type weights, with Rodrigues-type formulas provided either explicitly or in implicit form.

\

The purpose of this paper is to address this gap by enriching the literature with a diverse set of explicit and tractable nontrivial examples of bispectral matrix-valued orthogonal polynomials associated with discrete weights. These examples offer clear and practical illustrations to facilitate further exploration and applications in the theory of matrix orthogonality on discrete sets. To achieve this, we present a unified and explicit method for constructing a wide class of such polynomials associated with a weight matrix supported on a discrete set. The starting point is a collection of scalar discrete weights $w_{i}(x)$, $1\leq i \leq m$, supported on the same discrete set $\mathcal{J}$, together with their associated sequence of monic orthogonal polynomials $p_{n}^{w_{i}}(x)$. With these scalar weights, we construct a $m\times m$ weight matrix of the form 
$$W(x) = e^{Ax}\operatorname{diag}(w_{1}(x),\ldots,w_{m}(x))e^{A^{\ast}x},$$
where $A$ is a two-step nilpotent matrix. Then, we give an explicit closed formula for an associated sequence of matrix-valued orthogonal polynomials for $W$ in terms of $A$ and the scalar polynomials $p_{n}^{w_{i}}(x)$, see \eqref{Qn} and \eqref{Qn exp}.

A key feature of this construction is its generality: no additional assumptions are required on the scalar weights beyond the existence of their orthogonal polynomial sequences, and there is no need for $W$ to satisfy any Pearson-type equation or an extra structural condition. This grants significant flexibility, allowing the construction of many new examples not accessible via traditional approaches that rely on Pearson equations and Rodrigues-type formulas. Another advantage of our construction is that it provides a closed-form expression for each polynomial, in contrast to approaches based on Rodrigues formulas, where computing the $n$-th polynomial typically requires the recursive application of $n$ difference operators. This leads to expressions that are implicit, combinatorially intricate, and computationally demanding.

\

In addition to the generality and explicitness of the construction, a natural question is whether the resulting matrix-valued orthogonal polynomials also satisfy a difference equation, that is, whether they are bispectral. We address this in Theorem~\ref{bispectral}, where we provide a sufficient condition on the scalar polynomials $p_n^{w_i}(x)$ to ensure that the corresponding matrix-valued sequence is an eigenfunction of a second-order difference operator. This condition is then shown to hold in Theorem~\ref{classical bispectral} for all classical families of scalar discrete orthogonal polynomials (with a mild restriction on the parameters in the Hahn case). As a consequence, our construction yields bispectral matrix-valued extensions of all classical discrete families.

To the best of our knowledge, these provide the first explicit matrix-valued extensions of the Krawtchouk and Hahn families within this bispectral framework. We also introduce a novel bispectral family that combines distinct scalar discrete weights, specifically Charlier and Meixner weights. This highlights the versatility of our framework, as it allows for the construction of new families beyond the traditional single-weight setting. A similar strategy was used in our previous work \cite{BP25} to construct a matrix-valued bispectral family by combining Hermite and Laguerre weights in the continuous case.

\

Beyond their theoretical interest, these matrix-valued families could potentially be of interest in areas such as coding theory, discrete integrable systems, or matrix-valued random walks, where notions of orthogonality and bispectrality have appeared in various contexts. We expect that their explicit nature may facilitate future developments in these directions.

\

The paper is organized as follows. In Section \ref{section 2}, we present the necessary preliminaries, introducing the definition of a discrete weight matrix, the associated sequence of matrix-valued orthogonal polynomials, the algebra of difference operators, and recalling definitions and key properties of the classical discrete families. In Section \ref{section 3}, we develop the main tools used in this work, constructing the $m \times m$ weight matrix $W$ from a collection of scalar discrete weights. We then state Theorem \ref{main}, which provides an explicit expression for an orthogonal polynomial sequence for $W$, and Theorem \ref{bispectral}, which gives a sufficient condition ensuring that the constructed sequence is bispectral. In Section \ref{section 4}, we demonstrate that our construction extends all classical discrete scalar families to the matrix-valued setting while preserving bispectrality. That is, each such extension yields a sequence of $m \times m$ orthogonal polynomials that are eigenfunctions of a second-order matrix-valued difference operator. We conclude the section by presenting explicit $2 \times 2$ examples of Krawtchouk, Hahn, and mixed Charlier-Meixner type.

\section{Preliminaries} \label{section 2}
\subsection{Orthogonal polynomials and the algebra of difference operators $\mathcal{D}(W)$}

We aim to construct matrix-valued analogues of the classical discrete scalar families of Charlier, Meixner, Krawtchouk, and Hahn polynomials. Throughout this section, we fix some basic notation and introduce the type of matrix weights we will work with.

\

\begin{definition}
Let $\operatorname{Mat}_{m}(\mathbb{C})$ denote the space of $m \times m$ complex matrices. Let $\mathcal{J} \subset \mathbb{Z}$ be either $\mathbb{N}_0$ or a finite set of the form $\{0,1,\ldots,N\}$. A \emph{weight matrix} supported on a discrete set $\mathcal{J}$ is a function
\[
W:\mathbb{Z} \to \operatorname{Mat}_{m}(\mathbb{C})
\]
such that $W(x)$ is Hermitian positive definite for all $x \in \mathcal{J}$, vanishes for $x \notin \mathcal{J}$, and has finite moments of all orders, i.e.,
\[
\sum_{x \in \mathcal{J}} x^{n} W(x) < \infty \quad \text{for all } n \geq 0.
\]
\end{definition}

Given such a weight, we define a sesquilinear form on the space $\operatorname{Mat}_{m}(\mathbb{C})[x]$ of matrix-valued polynomials by
\[
\langle P, Q \rangle_{W} = \sum_{x \in \mathcal{J}} P(x) W(x) Q(x)^*, \quad \text{for all } P,Q \in \operatorname{Mat}_{m}(\mathbb{C})[x],
\]
where $^*$ denotes the conjugate transpose.

A sequence $\{P_n(x)\}$ ($n \in \mathbb{N}_{0}$, or $n = 0,1,\dots,N$ if $\mathcal{J}$ is a finite set) of $m \times m$ matrix-valued polynomials is said to be a sequence of orthogonal polynomials for $W$ if $\deg(P_n) = n$, the leading coefficient of $P_n(x)$ is invertible, and $\langle P_n, P_k \rangle_W = 0$ for all $n \ne k$. If the leading coefficient of each $P_n(x)$ is the identity matrix, we say that the sequence is \emph{monic}.

\begin{remark}
When the support $\mathcal{J}$ is infinite (i.e., $\mathcal{J} = \mathbb{N}_0$), any matrix-valued polynomial $P$ with nonsingular leading coefficient satisfies that $\langle P, P \rangle_{W}$ is invertible. This guarantees the existence of a unique sequence $\{P_n(x)\}_{n \geq 0}$ of monic orthogonal polynomials with respect to $W$.
\end{remark}

\begin{remark}
When the support is finite, say $\mathcal{J} = \{0,1,\ldots,N\}$, it is well known that only $N+1$ monic orthogonal polynomials exist. Indeed, any matrix-valued polynomial of degree greater than $N$ is linearly dependent on lower-degree polynomials when restricted to $\mathcal{J}$, since for any $k > N$, one can solve
\[
x^k = a_0 + x a_1  + \cdots + x^{N} a_N, \quad \text{for } x = 0,1,\ldots,N.
\]
This linear dependence prevents the existence of orthogonal polynomials beyond degree $N$.

However, in our construction of matrix-valued orthogonal polynomials, we require the $(N+1)$-th polynomial, particularly in \eqref{Qn}. Then, we adopt the natural extension
$$
p_{N+1}(x) := x(x-1)\cdots(x-N),
$$
which corresponds to the unique monic polynomial of degree $N+1$ that follows from the three-term recurrence relation satisfied by the orthogonal polynomials and is orthogonal with all lower-degree polynomials, see Proposition \ref{N+1} in the Appendix.
\end{remark}

By a standard argument (see \cite{K49} or \cite{K71}), one obtains that any sequence of orthogonal polynomials $\{P_n(x)\}$ with respect to $W$ satisfies a three-term recurrence relation of the form
\begin{equation}\label{three-term}
P_n(x) x = A_n P_{n+1}(x) + B_n P_n(x) + C_n P_{n-1}(x),
\end{equation}
for some matrices $A_n, B_n, C_n \in \operatorname{Mat}_m(\mathbb{C})$, where we adopt the convention $P_{-1}(x) = 0$.

The three-term recurrence relation defines a discrete operator $\mathcal{L} = A_{n}\mathscr{S} + B_{n} + C_{n} \mathscr{S}^{-1}$, where $\mathscr{S}^{k}$ acts on the left-hand side on a sequence as $\mathscr{S}^{k} \cdot p_{n} = p_{n+k}$, for $k \in \mathbb{Z}$. Thus, we have that 
$$\mathcal{L}\cdot P_{n}(x) = P_{n}(x)x.$$

\ 

Throughout this paper, we consider difference operators 
\begin{equation}\label{dif op}
    D = \sum_{j=0}^{s_{1}} \Delta^{j}F_{j}(x) + K(x) + \sum_{l=0}^{s_{2}}\nabla^{l} G_{l}(x),
\end{equation}
where 
\begin{equation*}
    \Delta(f(x)) = f(x+1)-f(x), \quad \nabla(f(x)) = f(x) - f(x-1),
\end{equation*}
and $F_{j},K,G_{l}$ are matrix-valued polynomials. These operators act on the right-hand side of matrix-valued functions as follows 
$$P(x) \cdot D = \sum_{j=0}^{s_{1}} \Delta^{j}(P(x))F_{j}(x) +  P(x)K(x) + \sum_{l=0}^{s_{2}}\nabla^{l}(P(x))G_{l}(x) .$$

From \cite{D12}, we introduce the definition of the algebra $\mathcal{D}(W)$. 

\begin{definition}
Given a weight matrix $W$ together with an associated sequence of orthogonal polynomials $P_{n}(x)$, the algebra $\mathcal{D}(W)$ is the algebra of all difference operators $D$ as defined in \eqref{dif op} that have the sequence $P_{n}(x)$ as eigenfunctions, i.e., the operators $D$ such that 
$$P_{n}(x) \cdot D = \Lambda_{n}(D) P_{n}(x),$$
for all $n$, with $\Lambda_{n}(D)\in \operatorname{Mat}_{m}(\mathbb{C})$.
\end{definition}

When $\mathcal{D}(W)$ contains a nontrivial operator $D$, we have together with the discrete operator of the three-term recurrence relation that 
\begin{equation*}
    \mathcal{L}\cdot P_{n}(x) = P_{n}(x)x, \quad \text{and} \quad P_{n}(x) \cdot D = \Lambda_{n}(D)P_{n}(x).
\end{equation*}
That is, the sequence $P_n(x)$ is simultaneously an eigenfunction of a left-hand side discrete operator with eigenvalue $x$, and of a right-hand side difference operator with eigenvalue $\Lambda_{n}(D)$. In this case, we say that the sequence ${P_n(x)}$ is \emph{bispectral}.
 
\subsection{Classical discrete scalar polynomials} \label{sub clas}
We recall from \cite{KS94} the classical scalar discrete polynomials together with their properties. We denote by $(a)_{n}$ the Pochhamer symbol, i.e, $(a)_{0} = 1$, $(a)_{n} = a(a+1)\cdots(a+n-1)$ for $n \in \mathbb{N}$, $a \in \mathbb{C}$.  
\subsubsection{Charlier Polynomials}

Let $b > 0$, the monic Charlier polynomials $C^{(b)}_{n}(x)$, are orthogonal with respect to the Charlier weight $w_{b}(x) = \frac{b^{x}}{x!}$ supported on $\mathbb{N}_{0}$. They satisfy the difference equation $C_{n}^{(b)}(x) \cdot \delta_{b} = \Lambda_{n}(\delta_{b})C_{n}^{(b)}(x)$, where:
$$\delta_{b} = (\Delta \, b  - \nabla \, x), \quad \text{ and } \quad \Lambda_{n}(\delta_{b}) =  -n$$
and are given by the Rodrigues formula, 
$$C^{(b)}_{n}(x) =(-b)^{n} \frac{x!}{b^{x}}\nabla^{n} \left ( \frac{b^{x}}{x!} \right ).$$
They satisfy the three-term recurrence relation
$$C_{n}^{(b)}(x)x = C^{(b)}_{n+1}(x) + (n+b)C^{(b)}_{n}(x) + nbC^{(b)}_{n-1}(x).$$
The squared norm is given by 
$$\|C_{n}^{(b)}(x)\|^{2} = n!e^{b}b^{n}.$$

\subsubsection{Meixner Polynomials}

Let $\beta > 0$ and $0 < c < 1$. The monic Meixner polynomials $M^{(\beta,c)}_{n}(x)$ are orthogonal with respect to the Meixner weight $w_{\beta,c}(x) = (\beta)_{x}\frac{c^{x}}{x!}$ supported on $\mathbb{N}_{0}$. They satisfy the difference equation $M_{n}^{(\beta,c)}(x) \cdot \delta_{\beta,c} = \Lambda_{n}(\delta_{\beta,c})M_{n}^{(\beta,c)}(x)$, where:
$$\delta_{\beta,c} = \Delta \, c(x+\beta) -  \nabla \, x \quad \text{ and } \quad \Lambda_{n}(\delta_{\beta,c}) =  n(c-1)$$
and are given by the Rodrigues formula, 
$$M^{(\beta,c)}_{n}(x) =(\beta)_{n} \frac{c^{n}}{(c-1)^{n}}\frac{x!}{(\beta)_{x}c^{x}}\nabla^{n} \left ( \frac{(\beta+n)_{x}c^{x}}{x!} \right ).$$
They satisfy the three-term recurrence relation
$$M_{n}^{(\beta,c)}(x)x = M_{n+1}^{(\beta,c)}(x) + \frac{n+(n+\beta)c}{1-c}M_{n}^{(\beta,c)}(x) + \frac{n(n+\beta-1)c}{(1-c)^{2}}M_{n-1}^{(\beta,c)}(x).$$
The squared norm is given by 
$$\|M_{n}^{(\beta,c)}(x)\|^{2} = (\beta)_{n}n!\frac{c^{n}}{(1-c)^{2n+\beta}}.$$

\subsubsection{Krawtchouk Polynomials}

Let $0 < p < 1$ and $N\in \mathbb{N}$. The monic Krawtchouk polynomials $K^{(p,N)}_{n}(x)$ are orthogonal with respect to the Krawtchouk weight $w_{p,N}(x) = \binom{N}{x} p^{x}(1-p)^{N-x}$ supported on $\{ 0 , 1 , \ldots , N \}$. They are given by the Rodrigues formula, 
$$K^{(p,N)}_{n}(x) =\frac{(-N)_{n}p^{n}}{\binom{N}{x} \left ( \frac{ p}{1-p} \right ) ^{x}}\nabla^{n} \left ( \binom{N-n}{x} \left ( \frac{ p}{1-p} \right )^{x}\right ).$$
The squared norm is given by 
$$\|K_{n}^{(p,N)}(x)\|^{2} = (-N)_{n}(-1)^{n}n!p^{n}(1-p)^{n}.$$
They satisfy the three-term recurrence relation
$$K_{n}^{(p,N)}(x)x = K_{n+1}^{(p,N)}(x) + (p(N-n)+n(1-p))K_{n}^{(p,N)}(x) + np(1-p)(N+1-n)K_{n-1}^{(p,N)}(x),$$
for $n = 0,\ldots, N$. 

The sequence $K_{n}^{(p,N)}$ satisfies the difference equation $K_{n}^{(p,N)}(x) \cdot \delta_{p,N} = \Lambda_{n}(\delta_{p,N})K_{n}^{(p,N)}(x)$, where:
$$\delta_{p,N} =  \Delta p(N-x) - \nabla x(1-p) \quad \text{ and } \quad  \Lambda_{n}(\delta_{p,N}) = -n.$$
\begin{remark}
    The polynomial $K_{N+1}^{(p,N)}(x) = x(x-1) \cdots (x-N)$ also satisfies the same difference equation, with the expected eigenvalue:
     $$K_{N+1}(x)^{(p,N)}(x) \cdot \delta_{p,N} = -(N+1)K_{N+1}^{(p,N)}(x).$$
\end{remark}

\subsubsection{Hahn Polynomials}
Let $N \in \mathbb{N}$, $\alpha,\beta > -1$, or $\alpha,\beta < -N$. The monic Hahn polynomials $H_{n}^{(\alpha,\beta,N)}(x)$ are orthogonal with respect to the Hahn weight $w_{\alpha,\beta,N}(x) = \binom{\alpha+x}{x}\binom{\beta+N-x}{N-x}$ supported on $\{0,1,\ldots,N\}.$ They are given by the Rodrigues formula,
$$H_{n}^{(\alpha,\beta,N)}(x) = \frac{(-1)^{n}(\alpha+1)_{n}(\beta+1)_{n}}{(n+\alpha+\beta+1)_{n}}\frac{1}{\binom{\alpha+x}{x}\binom{\beta+N-x}{N-x}}\nabla^{n}\left[\binom{\alpha+n+x}{x}\binom{\beta+N-x}{N-n-x}\right ].$$
The squared norm is given by
$$\|H_{n}^{(\alpha,\beta,N)}\|^{2} = (-1)^{n}\frac{(n+\alpha+\beta+1)_{N+1}}{(n+\alpha+\beta+1)_{n}^{2}}\frac{n!}{N!}\frac{(-N)_{n}(\alpha+1)_{n}(\beta+1)_{n}}{(2n+\alpha+\beta+1)}.$$
They satisfy the three-term recurrence relation
$$H_{n}^{(\alpha,\beta,N)}(x) x = H_{n+1}^{(\alpha,\beta,N)}(x) + (t_{n}+s_{n})H_{n}^{(\alpha,\beta,N)}(x)+t_{n-1}s_{n}H_{n-1}^{(\alpha,\beta,N)}(x), \quad \text{for }n =0,\dots,N,$$
where $t_{n} = \frac{(n+\alpha+\beta+1)(n+\alpha+1)(N-n)}{(2n+\alpha+\beta+1)(2n+\alpha+\beta+2)}$, and $s_{n} = \frac{n(n+\alpha+\beta+N+1)(n+\beta)}{(2n+\alpha+\beta)(2n+\alpha+\beta+1)}$. 

 They satisfy the difference equation $H_{n}^{(\alpha,\beta,N)}(x) \cdot \delta_{\alpha,\beta,N} = \Lambda_{n}(\delta_{\alpha,\beta,N})H_{n}^{(\alpha,\beta,N)}(x)$, where:
$$\delta_{\alpha,\beta,N} =  \Delta (x+\alpha+1)(x-N)-\nabla x(x-\beta-N-1) \quad \text{and} \quad \Lambda_{n}(\delta_{\alpha,\beta,N}) = n(n+\alpha+\beta+1).$$
\begin{remark}
    The polynomial $H_{N+1}^{(\alpha,\beta,N)}(x) = x(x-1) \cdots (x-N)$ also satisfies the same difference equation, with the expected eigenvalue:
     $$H_{N+1}^{(\alpha,\beta,N)}(x) \cdot \delta_{\alpha,\beta,N} = -(N+1)(N+\alpha+\beta+2)H_{N+1}^{(\alpha,\beta,N)}(x).$$
\end{remark}

\section{Construction of bispectral matrix discrete polynomials}\label{section 3}
In this section, we introduce the discrete weight matrices, together with an explicit expression of the associated sequence of orthogonal polynomials. We also give a condition that ensures that the constructed sequence is bispectral. 

\ 

Let $w_{1},w_{2},\ldots,w_{m}$ be scalar weights supported on the same discrete set $\mathcal{J}$ ($\mathbb{N}_{0}$ or $\{0,1,\ldots,N\}$). We define the $m\times m$ weight matrix given by 
\begin{equation}\label{W}
    W(x) = T(x)\tilde{W}(x)T(x)^{\ast}, \quad x \in \mathcal{J},
\end{equation}
where $\tilde{W}(x) = \operatorname{diag}(w_{1}(x),\ldots,w_{m}(x))$, and $T(x) = e^{Ax} = I + Ax$, with $A$ the two-step nilpotent matrix defined by 
\begin{equation}\label{A}
A = \sum_{j=1}^{[m / 2]} a_{2j-1} E_{2j-1, 2j} + \sum_{j=1}^{[(m-1)/2]} a_{2j} E_{2j+1, 2j}, \quad a_{j}\in\mathbb{R}\setminus\{0\}.
\end{equation}
For $m = 2$, and $m = 3$ we have respectively
\begin{equation*}
        W(x) = \begin{pmatrix} 1 & ax \\ 0 & 1 \end{pmatrix} \begin{pmatrix} w_{1}(x) & 0 \\ 0 & w_{2}(x) \end{pmatrix} \begin{pmatrix} 1 & 0 \\ ax & 1 \end{pmatrix} =  \begin{pmatrix}w_{1}(x) + w_{2}(x)a^{2}x^{2} & w_{2}(x)ax \\w_{2}(x)ax & w_{2}(x) \end{pmatrix}, 
\end{equation*}
and
\begin{equation*}
    \begin{split}
            W(x) & = \begin{pmatrix} 1 & a_{1}x & 0 \\ 0 & 1 & 0 \\ 0 & a_{2}x & 1 \end{pmatrix} \begin{pmatrix} w_{1}(x) & 0 & 0 \\ 0 & w_{2}(x) & 0 \\ 0 & 0 & w_{3}(x) \end{pmatrix} \begin{pmatrix} 1 & 0 & 0 \\ a_{1}x & 1 & a_{2}x \\ 0 & 0 & 1 \end{pmatrix} \\
            & = \begin{pmatrix}w_{1}(x) + w_{2}(x)a_{1}^{2}x^{2} & w_{2}(x)a_{1}x  & w_{2}(x)a_{1}a_{2}x^{2} \\ w_{2}(x)a_{1}x & w_{2}(x) & w_{2}(x)a_{2}x \\ w_{2}(x)a_{1}a_{2}x^{2} & w_{2}(x) a_{2}x & w_{3}(x) + w_{2}(x)a_{2}^{2}x^{2}  \end{pmatrix}.
    \end{split}
\end{equation*}
We denote by $p_{n}^{w_{i}}(x)$ the $n$-th monic orthogonal polynomial for the scalar weight $w_{i}(x)$. Therefore, we have that $P_{n}(x) = \operatorname{diag}(p_{n}^{w_{1}}(x),\ldots,p_{n}^{w_{m}}(x))$ is the sequence of monic orthogonal polynomials for the diagonal weight $\tilde{W}(x) = \operatorname{diag}(w_{1}(x),\ldots,w_{m}(x))$. We have the square norm given by 
\begin{equation*}
    \begin{split}
        \|P_{n}\|^{2} = \langle P_{n}, P_{n} \rangle_{\tilde{W}} &=\sum_{j\in \mathcal{J}}P_{n}(x)\operatorname{diag}(w_{1}(x),\ldots,w_{m}(x))P_{n}(x)^{\ast} \\
        & = \operatorname{diag}(\|p_{n}^{w_{1}}\|^{2},\ldots, \|p_{n}^{w_{m}}\|^{2}),
    \end{split}
\end{equation*}
which is an invertible matrix for all $n \in \mathbb{N}_{0}$ if the support $\mathcal{J}$ is infinite, and for all $0 \leq n \leq N$ if $\mathcal{J}$ is finite. In the finite-support case, we recall that we extend the scalar sequence by setting $p_{N+1}^{w_{i}}(x) = x(x-1)\cdots(x-N)$; in that case, the corresponding norm vanishes, that is, $\|P_{N+1}\|^{2} = 0$.
Now, we state our main theorem.
 
\smallskip

\medskip
\begin{thm}\label{main}
Let $W(x) = T(x)\tilde{W}(x)T(x)^{\ast}$ be as defined in \eqref{W}, let $A$ be as in \eqref{A}, and let $P_{n}(x) = \operatorname{diag}(p_{n}^{w_{1}}(x),\dots,p_{n}^{w_{m}}(x))$ denote the sequence of monic orthogonal polynomials for $\tilde{W}$. Then,
\begin{equation}\label{Qn}
\begin{split}
Q_n(x) &= P_n(x) + A P_{n+1}(x) - \|P_n\|^2 A^{\ast} \|P_{n-1}\|^{-2} P_{n-1}(x) \\
&\quad - P_n(x)Ax + \|P_n\|^2 A^{\ast} \|P_{n-1}\|^{-2} P_{n-1}(x)Ax
\end{split}
\end{equation}
is a sequence of orthogonal polynomials for $W$ for all $n \in \mathbb{N}_0$ if the support is infinite, and for $n = 0, 1, \ldots, N$ if the support is finite. We adopt the convention that $P_{-1}(x)=0$, so that for $n=0$ the formula yields 
$$
Q_0(x) = P_0(x) + A P_1(x) - P_0(x) A x.
$$
\end{thm}
\begin{proof}
    Following an argument similar to the one given in the proof of Theorem 3.1 in \cite{BP25}, one can show that $Q_{n}(x)$ is a polynomial of degree $n$ with a nonsingular leading coefficient. We will prove the orthogonality. First, we define the set $\mathcal{M}$ by
 \begin{equation}\label{set M}
 \mathcal{M} = \left \{  M = \sum_{j=1}^{[N/2]} m_{2j-1}E_{2j-1,2j} + \sum_{j=1}^{[(N-1)/2]}m_{2j}E_{2j+1,2j}, \quad m_{j} \in \mathbb{C} \right \}.
 \end{equation}
We note that the two-step nilpotent matrix $A$ belongs to $\mathcal{M}$. For any diagonal matrix $D = \operatorname{diag}(d_{1},\ldots,d_{m})$, and any matrix $M \in \mathcal{M}$, we have that $DM$ and $MD$ belong to $\mathcal{M}$. We also have that $M_{1}M_{2} = 0 = M_{2}M_{1}$ for every $M_{1},M_{2} \in \mathcal{M}.$

From the above observations, it follows that 
\begin{equation} \label{QnT}
    Q_{n}(x)T(x) = Q_{n}(x)e^{Ax}=Q_{n}(x)(I+Ax) = P_{n}(x) +AP_{n+1}(x) - \|P_{n}\|^{2}A^{\ast}\|P_{n-1}\|^{-2}P_{n-1}(x).
\end{equation}
 
Let $n\not= m$. Since $W(x) = T(x)\tilde{W}(x)T(x)^{\ast}$, with $\tilde{W}(x) = \operatorname{diag}(w_{1}(x),\ldots,w_{m}(x))$, we have that 
\begin{equation*}
        \begin{split}
            \langle Q_{n},Q_{m} \rangle_{W} & = \langle Q_{n}T, Q_{m}T \rangle_{\tilde{W}}  \\
            & = \langle P_{n}, P_{n}\rangle_{\tilde{W}}+ A\langle P_{n+1}, P_{m} \rangle_{\tilde{W}}-\langle P_{n}, P_{m-1}\rangle_{\tilde{W}} \|P_{m-1}\|^{-2}A\|P_{m}\|^{2} \\ 
            & \quad + \langle P_{n},P_{m+1} \rangle_{\tilde{W}} A^{\ast}-\|P_{n}\|^{2}A^{\ast}\|P_{n-1}\|^{-2}\langle P_{n-1}, P_{m}\rangle_{\tilde{W}} + A\langle P_{n+1}, P_{m+1} \rangle_{\tilde{W}} A^{\ast} \\
            & \quad  + \|P_{n}\|^{2}A^{\ast}\|P_{n-1}\|^{-2}\langle P_{n-1}, P_{m-1}\rangle_{\tilde{W}} \|P_{m-1}\|^{-2}A\|P_{m}\|^{2}.
        \end{split}
    \end{equation*}
From here, by the orthogonality of $P_{n}$ with respect to $\langle \cdot \, , \cdot \rangle_{\tilde{W}}$, we obtain that
    \begin{equation} \label{inner}
        \begin{split}
            \langle Q_{n},Q_{m}\rangle_{W} & = A\langle P_{n+1},P_{m} \rangle_{\tilde{W}} - \langle P_{n}, P_{m-1} \rangle_{\tilde{W}}\|P_{m-1}\|^{-2}A\|P_{m}\|^{2} \\
            & \quad + \langle P_{n},P_{m+1} \rangle_{\tilde{W}} A^{\ast} - \|P_{n}\|^{2}A^{\ast}\|P_{n-1}\|^{-2}\langle P_{n-1},P_{m} \rangle_{\tilde{W}}.
        \end{split}
    \end{equation}
Now, we need to consider three cases: $n = m+1$, $n = m-1$, and $n \neq m+1$ and $n \neq m-1$. In all three cases, it follows immediately that \eqref{inner} is equal to $0$. Finally, for $n=0$, it is clear that $Q_0$ is orthogonal to $Q_n$ for all $n \ge 1$.
\end{proof}

Explicitly, the expression of \eqref{Qn} is given by
{\footnotesize
\begin{equation}\label{Qn exp}
    \begin{split}
        Q_{n} & = \begin{psmallmatrix} p_{n}^{w_{1}} &  a_{1}(p_{n+1}^{w_{2}} - p_{n}^{w_{1}} \, x) & 0 & 0 & 0 & 0  & \cdots \\ -a_{1}\|p_{n}^{w_{2}}\|^{2} \frac{p_{n-1}^{w_{1}}}{\|p_{n-1}^{w_{1}}\|^{2}} & p_{n}^{w_{2}} & -a_{2}\|p_{n}^{w_{2}}\|^{2} \frac{p_{n-1}^{w_{3}}}{\|p_{n-1}^{w_{3}}\|^{2}} & 0 & 0 & 0 & \\ 0 & a_{2}(p_{n+1}^{w_{2}} - p_{n}^{w_{3}} \, x ) & p_{n}^{w_{3}} & a_{3}(p_{n+1}^{w_{4}} - p_{n}^{w_{3}} \, x ) & 0 & 0 & \cdots \\ 0 & 0 & -a_{3}\|p_{n}^{w_{4}}\|^{2} \frac{p_{n-1}^{w_{3}}}{\|p_{n-1}^{w_{3}}\|^{2}} & p_{n}^{w_{4}} & -a_{4}\|p_{n}^{w_{4}}\|^{2} \frac{p_{n-1}^{w_{5}}}{\|p_{n-1}^{w_{5}}\|^{2}} & 0 &  \\ 0 & 0 & 0 & a_{4}(p_{n+1}^{w_{4}} - p_{n}^{w_{5}} \, x) & p_{n}^{w_{5}} & a_{5}(p_{n+1}^{w_{6}}-p_{n}^{w_{5}} \, x) & \cdots  \\ 0 & 0 & 0 & 0 & -a_{5}\|p_{n}^{w_{6}}\|^{2} \frac{p_{n-1}^{w_{5}}}{\|p_{n-1}^{w_{5}}\|^{2}} & p_{n}^{w_{6}} & \ddots \\ \vdots & \vdots & \vdots &  \vdots & \vdots & \ddots & \ddots\end{psmallmatrix} \\ 
        & + \begin{psmallmatrix} 0 & 0 & 0 & 0 & 0 & 0 & \cdots \\ 0 & \|p_{n}^{w_{2}}\|^{2}\left (a_{1}^{2}\frac{p_{n-1}^{w_{1}}}{\|p_{n-1}^{w_{1}}\|^{2}} + a_{2}^{2}\frac{p_{n-1}^{w_{3}}}{\|p_{n-1}^{w_{3}}\|^{2}}\right ) & 0 & a_{2}a_{3}\|p_{n}^{w_{2}}\|^{2}\frac{p_{n-1}^{w_{3}}}{\|p_{n-1}^{w_{3}}\|^{2}} \, x & 0 & 0 & \cdots \\ 0 & 0 & 0 & 0 & 0 & 0 & \cdots \\ 0 & a_{2}a_{3}\|p_{n}^{w_{4}}\|^{2}\frac{p_{n-1}^{w_{3}}}{\|p_{n-1}^{w_{3}}\|^{2}} \, x & 0 & \|p_{n}^{w_{4}}\|^{2}\left (a_{3}^{2}\frac{p_{n-1}^{w_{3}}}{\|p_{n-1}^{w_{3}}\|^{2}} + a_{4}^{2}\frac{p_{n-1}^{w_{5}}}{\|p_{n-1}^{w_{5}}\|^{2}}\right ) & 0 & a_{4}a_{5}\|p_{n}^{w_{4}}\|^{2}\frac{p_{n-1}^{w_{5}}}{\|p_{n-1}^{w_{5}}\|^{2}} \, x & \cdots \\ 0 & 0 & 0 & 0 & 0 & 0  & \ddots \\ 0 & 0 & 0 & a_{4}a_{5}\|p_{n}^{w_{6}}\|^{2}\frac{p_{n-1}^{w_{5}}}{\|p_{n-1}^{w_{5}}\|^{2}} \, x  & 0 & \|p_{n}^{w_{6}}\|^{2}\left (a_{5}^{2}\frac{p_{n-1}^{w_{5}}}{\|p_{n-1}^{w_{5}}\|^{2}} + a_{6}^{2}\frac{p_{n-1}^{w_{7}}}{\|p_{n-1}^{w_{7}}\|^{2}}\right ) & \ddots \\ \vdots & \vdots  &  \vdots & \vdots & \ddots & \ddots & \ddots \end{psmallmatrix}.
    \end{split}
\end{equation}
}

\

We now turn to the question of bispectrality. While the previous theorem provides an explicit orthogonal sequence for any choice of scalar weights, it is natural to ask under what conditions this sequence also satisfies a difference equation, thus yielding bispectral polynomials. The following result addresses this question by giving a sufficient condition on the scalar weights to ensure that the matrix-valued polynomials are eigenfunctions of a second-order difference operator.

\begin{thm}\label{bispectral}
    Let $w_{i}$ be scalar discrete weights and $p_{n}^{w_{i}}$ their monic orthogonal polynomials, $1 \leq i \leq m$. Suppose each $p_{n}^{w_{i}}$ is an eigenfunction of a second-order difference operator $\delta_{i}$ with eigenvalue $\Lambda_{n}(\delta_{i})$ such that the condition 
    \begin{equation}\label{condition}
        \Lambda_{n}(\delta_{i}) = \Lambda_{n+1}(\delta_{j}) \quad \text{for all odd } i \text{ and even } j
    \end{equation}
    is satisfied. Then, the matrix polynomials $Q_{n}$ constructed in Theorem \ref{main} are eigenfunctions of a second-order difference operator.
\end{thm}
\begin{proof}
    By hypothesis, there exists a second-order difference operator $\delta_{i}  \in \mathcal{D}(w_{i})$, $1\leq i \leq m$. Let $P_{n}(x) = \operatorname{diag}(p_{n}^{w_{i}}(x),\ldots, p_{n}^{w_{m}}(x))$, and $\tilde{D} = \operatorname{diag}(\delta_{1},\dots,\delta_{m})$. Then, we have that 
    $$P_{n}(x) \cdot D = \Lambda_{n}(\tilde{D})P_{n}(x),$$
    where $\Lambda_{n}(\tilde{D}) = \operatorname{diag}(\Lambda_{n}(\delta_{1}),\dots,\Lambda_{n}(\delta_{m}))$. Let $A$ be the two-step nilpotent matrix as defined in \eqref{A}, and $T(x) = e^{Ax}= I + Ax$. We show that the sequence $Q_{n}$ is an eigenfunction of the difference operator $T\tilde{D}T^{-1}$. From \eqref{QnT}, we obtain 
    \begin{equation*}
        \begin{split}
            Q_{n}(x)\cdot D & = (P_{n}(x)+AP_{n+1}(x)-\|P_{n}\|^{2}A^{\ast}\|P_{n-1}\|^{-2}P_{n-1}(x))\tilde{D}T(x)^{-1} \\
            & = (\Lambda_{n}(\tilde{D})P_{n}(x)+A\Lambda_{n+1}(\tilde{D})P_{n+1}(x)-\|P_{n}\|^{2}A^{\ast}\|P_{n-1}\|^{-2}\Lambda_{n-1}(\tilde{D})P_{n-1}(x))T(x)^{-1}.
        \end{split}
    \end{equation*}
    By the condition in \eqref{condition}, for any matrix $M \in \mathcal{M}$ (where $\mathcal{M}$ is the set defined in \eqref{set M}), we have $\Lambda_{n}(\tilde{D}) M = M \Lambda_{n+1}(\tilde{D})$. Thus, the above equation becomes
    \begin{equation*}
        \begin{split}
            Q_{n}(x)\cdot D & = (\Lambda_{n}(\tilde{D})P_{n}(x) + \Lambda_{n}(\tilde{D})AP_{n+1}(x) - \Lambda_{n}(\tilde{D})\|P_{n}\|^{2}A^{\ast}\|P_{n-1}\|^{-2}P_{n-1}(x))T(x)^{-1} \\
            & = \Lambda_{n}(\tilde{D})Q_{n}(x)T(x)T(x)^{-1} = \Lambda_{n}(\tilde{D})Q_{n}(x).
        \end{split}
    \end{equation*}
    Hence, the statement holds.
\end{proof}
\begin{cor}
    Under the assumptions of the above theorem, assume further that each difference operator $\delta_i$ has the form
    $$
    \delta_i = \Delta f_i(x) + k_i(x) - \nabla g_i(x),
    $$
    where $f_i(x)$, $k_i(x)$, and $g_i(x)$ are scalar polynomials. Then, the sequence $Q_n$ constructed in Theorem~\ref{main} is an eigenfunction of the second-order matrix-valued difference operator
    \begin{equation}\label{difference}
    \begin{split}
        D &= \Delta \big( (I+A)F(x) + [A,F(x)]\,x \big) + A\big(F(x) - G(x)\big) + K(x) + [A,K(x)]\,x \\
        &\quad - \nabla\big( (I-A)G(x) + [A,G(x)]\,x \big),
    \end{split}
    \end{equation}
    where $A$ is the nilpotent matrix defined in \eqref{A}, and
    $$
    F(x) = \operatorname{diag}(f_1(x), \ldots, f_m(x)), \quad 
    K(x) = \operatorname{diag}(k_1(x), \ldots, k_m(x)), \quad 
    G(x) = \operatorname{diag}(g_1(x), \ldots, g_m(x)).
    $$
    Moreover, we have
    $$
    Q_n(x) \cdot D = \operatorname{diag}(\Lambda_n(\delta_1), \ldots, \Lambda_n(\delta_m)) \, Q_n(x).
    $$
\end{cor}

\begin{proof}
    The result follows by explicitly computing the conjugation
    $$
    D = T(x)\, \operatorname{diag}(\delta_1, \ldots, \delta_m)\, T(x)^{-1},
    $$
    where $T(x) = e^{Ax} = I + Ax$, and using the proof of the above theorem.
\end{proof}

\

\section{Classical Bispectral Matrix Orthogonal Polynomials}\label{section 4}

In this section, we show that all classical families of scalar discrete orthogonal polynomials can be extended to the matrix-valued setting in a way that preserves bispectrality. More precisely, using Theorems~\ref{main} and~\ref{bispectral}, we construct matrix-valued orthogonal polynomials of arbitrary size $m \times m$ that are eigenfunctions of second-order difference operators whenever the scalar weights belong to classical families. Notably, the construction also allows for mixing different types of scalar weights—such as combining Charlier and Meixner weights—yielding new bispectral matrix-valued families beyond the single-family extensions. In the following theorem, we establish the bispectrality of the resulting matrix-valued extensions.

\begin{thm}\label{classical bispectral}
Let $w_1(x), \dots, w_m(x)$ be scalar discrete weights, all belonging to one of the following families:
\begin{itemize}
    \item[\textbf{(a)}] \textbf{Charlier-type:} $w_i(x) = \dfrac{b_i^x}{x!}$ with $b_i > 0$;

    \
    
    \item[\textbf{(b)}] \textbf{Meixner-type:} $w_i(x) = (\beta_i)_x \dfrac{c_i^x}{x!}$ with $\beta_i > 0$ and $0 < c_i < 1$;

    \
    
    \item[\textbf{(c)}] \textbf{Mixed Charlier–Meixner type:} each $w_i$ is either as in (a) or (b);

    \
    
    \item[\textbf{(d)}] \textbf{Krawtchouk-type:} $w_i(x) = \binom{N}{x} p_i^x (1 - p_i)^{N - x}$ with $0 < p_i < 1$ and fixed $N \in \mathbb{N}$;

    \
    
    \item[\textbf{(e)}] \textbf{Hahn-type:} $w_i(x) = \binom{\alpha_i + x}{x} \binom{\beta_i + N - x}{N - x}$ with $N \in \mathbb{N}$ and $\alpha_i, \beta_i > -1$ or $\alpha_{i},\beta_{i} < -N$, satisfying
    \begin{equation}\label{hahn cond}
    \alpha_i + \beta_i = \alpha_j + \beta_j + 2 \quad \text{for all odd } i \text{ and even } j.
    \end{equation}
\end{itemize}

Then, the sequence of matrix-valued orthogonal polynomials for the weight matrix  $W(x) = T(x)\operatorname{diag}(w_{1}(x)\dots,w_{m}(x))T(x)^{\ast}$, as constructed in \eqref{W}, is an eigenfunction of a second-order difference operator.
\end{thm}

\begin{proof}
The second-order difference operators
$$
\delta_{b_i} = \Delta b_i - \nabla x, \quad 
\delta_{\beta_i,c_i} = \Delta \, c_{i}(x+\beta_{i}) - \nabla x, \quad 
\delta_{p_i,N} = \Delta \, p_{i}(N-x) - \nabla \, x(1-p_{i})
$$
associated to the Charlier, Meixner, and Krawtchouk weights, respectively, belong to the algebra $\mathcal{D}(w_i)$ and satisfy
$$
\Lambda_n(-\delta_{b_i}) = n, \quad 
\Lambda_n\left( \frac{\delta_{\beta_i,c_i}}{c_i - 1} \right) = n, \quad 
\Lambda_n(-\delta_{p_i,N}) = n.
$$
Hence, in each case we obtain an operator in $\mathcal{D}(w_i)$ whose eigenvalue is $n$. By adding $1$ to such operators when $i$ is odd, and leaving it unchanged when $i$ is even, we construct new operators $\delta_i$ such that the condition \eqref{condition} holds. Then, by Theorem \ref{bispectral}, the resulting matrix-valued orthogonal polynomials are eigenfunctions of a second-order difference operator.

\

For the Hahn case, we have the second-order difference operator  
$$
\delta_{\alpha_i,\beta_i,N} = \Delta (x + \alpha_i + 1)(N - x)  - \nabla x(\beta_i + N - x + 1), \quad \text{with} \quad \Lambda_{n}(\delta_{\alpha_{i},\beta_{i},N}) = -n(n+\alpha_{i}+\beta_{i}+1).
$$
Then, we take $\delta_{i} = \delta_{\alpha_{i},\beta_{i},N}$ if $i$ is odd, and $\delta_{i} = \delta_{\alpha_{i},\beta_{i},N} - \alpha_{1}-\beta_{1}$ if $i$ is even, for all $1 \leq i \leq m$. Then, by using the relation on the parameters given in \eqref{hahn cond}, we obtain that $\Lambda_{n}(\delta_{i}) = \Lambda_{n+1}(\delta_{j})$ for all odd $i$ and even $j$, $1 \leq i,j \leq m$. Thus, by Theorem \ref{bispectral}, the statement holds.
\end{proof}

In the following subsections, we explicitly show some $2\times 2$ examples.

\subsection{Matrix-valued Krawtchouk}
Let $0 < p,s < 1$, $N \in \mathbb{N}$, and $a \in \mathbb{R} \setminus \{0 \}$. Consider the scalar Krawtchouk weights $w_{p,N}(x) = \binom{N}{x}p^{x}(1-p)^{N-x}$, and $w_{s,N}(x) = \binom{N}{x}s^{x}(1-s)^{N-x}$ supported on $\{0,1,\dots,N\}$. As in \eqref{W}, we construct the $2\times 2$ weight matrix
\begin{equation*}\label{krawt weight}
    \begin{split}
        W_{p,s,a,N}(x) & = \begin{pmatrix} 1 & ax \\ 0 & 1 \end{pmatrix} \begin{pmatrix} w_{p,N}(x) & 0 \\ 0 & w_{s,N}(x) \end{pmatrix} \begin{pmatrix} 1 & 0 \\ ax & 1 \end{pmatrix} \\ 
        & = \binom{N}{x}
\begin{pmatrix}
p^x (1-p)^{N-x} + a^2 x^2 s^x (1-s)^{N-x} & a x s^x (1-s)^{N-x} \\
a x s^x (1-s)^{N-x} & s^x (1-s)^{N-x}
\end{pmatrix}
    \end{split}
\end{equation*}
supported on $\{0, 1, \ldots, N\}$.

By Theorem \ref{main}, a sequence $\{Q_n^{p,s,N,a}(x)\}_{n=0}^N$ of orthogonal polynomials for $W_{p,s,a,N}$ is given explicitly by
\begin{equation*}\label{krawt 2}
Q_{n}^{p,s,N,a}(x) = 
\begin{psmallmatrix}
K_{n}^{(p,N)}(x) && a\left(K_{n+1}^{(s,N)}(x) - K_n^{(p,N)}(x) x \right) \\
- a n \frac{(1-s)^{n} s^{n} (N - n + 1)}{p^{n-1}(1-p)^{n-1}} K_{n-1}^{(p,N)}(x) && 
a^2 n \frac{(1-s)^{n} s^{n}(N - n + 1)}{p^{n-1}(1-p)^{n-1}} K_{n-1}^{(p,N)}(x) x + K_n^{(s,N)}(x)
\end{psmallmatrix},
\end{equation*}
where $K_n^{(p,N)}(x)$ and $K_{n}^{(s,N)}(x)$ are the monic orthogonal Krawtchouk polynomials for $w_{p,N}(x)$ and $w_{s,N}(x)$, respectively. We recall that we are taking the $(N+1)$-th polynomial of Krawtchouk as $K_{N+1}^{(s,N)}(x) = x(x-1)\cdots(x-N)$.

The polynomials $Q_{n}^{p,s,N,a}$ are eigenfunctions of the difference operator
\begin{equation*}
D = \Delta \begin{psmallmatrix} -p(N - x) & a(x(p - s) - s)(N - x) \\ 0 & -s(N - x) \end{psmallmatrix} + \begin{psmallmatrix}1 & -a N s \\ 0 & 0 \end{psmallmatrix} - \nabla  \begin{psmallmatrix} -x(1 - p) & -a x(x(p - s) + s - 1) \\ 0 & -x(1 - s) \end{psmallmatrix}.
\end{equation*}
We have that
$$Q_{n}^{p,s,N,a}(x) \cdot D = \begin{pmatrix}n+1 & 0 \\0 & n \end{pmatrix}Q_{n}^{p,s,N,a}(x).
$$

By using the explicit expression of the sequence $\{Q_{n}^{p,s,N,a}(x)\}$ and the three-term recurrence relation satisfied by the scalar Krawtchouk polynomials, one can directly verify that $Q_n^{p,s,N,a}(x)$ satisfies a three-term recurrence relation of the form
$$
Q_{n}^{p,s,N,a}(x)\, x = A_{n} Q_{n+1}^{p,s,N,a}(x) + B_{n} Q_{n}^{p,s,N,a}(x) + C_{n} Q_{n-1}^{p,s,N,a}(x),
$$
where
\begin{equation*}
\begin{split}
A_n &= 
\begin{psmallmatrix}
1 & 
\frac{-a n p(p - 1)(N + 1 - n)\left(N (p - s) - 2n (p - s) + 2s - 1\right)}{\theta(n)} \\
0 & 
\frac{n p(p - 1)(N + 1 - n)(\mu_n a^2 + 1)}{\theta(n)}
\end{psmallmatrix}, \\
B_n &= 
\begin{psmallmatrix}
\frac{\theta(n)((N-2(n+1))s+n+1)+np(p-1)(N+1-n)(N(p-s)-2n(p-s)+2s-1)}{\theta(n)} & \frac{a((N - n)(p - s)(n(p + s - 1) + s) + (p - 1)(n(p + s) - N s))}{\mu_n a^2 + 1} \\
\frac{\mu_{n}a((n-N)(n(p-s)(p+s-1)-s^{2}+s)+n(p-p^{2}))}{\theta(n)} & \frac{(N - 2n)(\mu_n a^2 p + s) + \mu_n a^2(2p + n - 1) + n}{\mu_n a^2 + 1} 
\end{psmallmatrix}, \\
C_n &= 
\begin{psmallmatrix}
-\frac{\mu_n a^2 s(s - 1)(n + 1)(N - n) + n p(p - 1)(N + 1 - n)}{\mu_n a^2 + 1} & 0 \\
-\frac{\mu_n a (N p - N s - 2n p + 2n s + 2p - 1)}{\mu_n a^2 + 1} & -n s(s - 1)(N + 1 - n)
\end{psmallmatrix},
\end{split}
\end{equation*}
where
\begin{equation*}
    \begin{split}
        \mu_{n} & = n \left ( \frac{1-s}{1-p} \right ) ^{n-1} \left ( \frac{s}{p} \right)^{n-1}(1-s)s(N-n+1), \\
        \theta(n) & = \mu_n a^2 s (s - 1) (n + 1) (N - n) + n p (p - 1) (N + 1 - n).
    \end{split}
\end{equation*}

To illustrate the construction, we compute the five polynomials for $p = s = \frac{1}{2}$, $N = 4$. We have

$$W(x) = \binom{4}{x}\frac{1}{16}\begin{pmatrix}1 + a^{2}x^{2} & ax \\ax & 1 \end{pmatrix},$$

\begin{equation*}
    \begin{split}
        Q_{0}(x) & = \begin{psmallmatrix} 1 & -2a \\ 0 & 1 \end{psmallmatrix},\quad Q_1(x) = \begin{psmallmatrix} x-2 & -a(2x-3) \\ -a & a^2x + x-2 \end{psmallmatrix}, \quad Q_{2}(x) = \begin{psmallmatrix} 
(x-1)(x-3) & -\frac{1}{2}a(4x^2-13x+6) \\ 
-\frac{3}{2}a(x-2) & \frac{3}{2}a^2x^2 - 3a^2x + x^2 -4x +3 
\end{psmallmatrix}, \\
        Q_{3}(x) & = \begin{psmallmatrix}
\left(\frac{1}{2}x -1\right)(2x^2 -8x +3) & -\frac{1}{2}a(4x^3 -21x^2 +26x -3) \\
-\frac{3}{2}a(x-1)(x-3) & 
  \frac{3}{2}a^2x^3 -6a^2x^2 + \frac{9}{2}a^2x + x^3 -6x^2 + \frac{19}{2}x -3
\end{psmallmatrix}, \\
        Q_{4}(x) & = \begin{psmallmatrix}
x^4 -8x^3 +20x^2 -16x + \frac{3}{2} & -\frac{1}{2}ax(2x-5)(2x^2 -10x +9) \\
-\frac{1}{2}a(x-2)(2x^2 -8x +3) & 
  a^2x^4 -6a^2x^3 + \frac{19}{2}a^2x^2 -3a^2x + x^4 -8x^3 +20x^2 -16x + \frac{3}{2}
\end{psmallmatrix}.
    \end{split}
\end{equation*}

\subsection{Matrix-valued Hahn} 
We now present a matrix-valued version of the Hahn polynomials. Let $N \in \mathbb{N}$ and parameters $\alpha, \beta, \tilde{\alpha}, \tilde{\beta}$ such that $\alpha, \beta, \tilde{\alpha}, \tilde{\beta} > -1$ or $\alpha, \beta, \tilde{\alpha}, \tilde{\beta} < -N$. While no other conditions on these parameters are strictly necessary to construct a matrix-valued Hahn polynomial using Theorem \ref{main}, we impose the additional constraint $\alpha + \beta = \tilde{\alpha} + \tilde{\beta} + 2$ to ensure the resulting matrix polynomials are eigenfunctions of a difference operator. (Interestingly, this same condition on the parameters also arises in the continuous setting when constructing matrix-valued Jacobi polynomials that are eigenfunctions of a second-order differential operator; see \cite{BP25}.)

Let $H^{(\alpha,\beta,N)}_{n}(x)$ and $H^{(\tilde{\alpha},\tilde{\beta},N)}_{n}(x)$ be the monic Hahn polynomials orthogonal with respect to the weights $w_{1}(x) = \binom{\alpha+x}{x}\binom{\beta+N-x}{N-x}$ and $w_{2}(x) = \binom{\tilde{\alpha}+x}{x}\binom{\tilde{\beta}+N-x}{N-x}$, respectively. We construct the weight matrix as in \eqref{W}, for $a \in \mathbb{R} \setminus \{0\}$ and $x \in \{0,1,\ldots,N\}$ we have 
\begin{equation*}\label{Hahn weight}
W_{\alpha,\beta,\tilde{\alpha},\tilde{\beta},a}(x) = \begin{pmatrix}
\binom{\alpha+x}{x}\binom{\beta+N-x}{N-x} + a^{2}x^{2} \binom{\tilde{\alpha}+x}{x}\binom{\tilde{\beta}+N-x}{N-x} & ax\binom{\tilde{\alpha}+x}{x}\binom{\tilde{\beta}+N-x}{N-x} \\
ax\binom{\tilde{\alpha}+x}{x}\binom{\tilde{\beta}+N-x}{N-x} & \binom{\tilde{\alpha}+x}{x}\binom{\tilde{\beta}+N-x}{N-x}
\end{pmatrix}.
\end{equation*}
Now, by Theorem \ref{main}, we obtain an explicit sequence of orthogonal polynomials for $W_{\alpha,\beta,\tilde{\alpha},\tilde{\beta},a}$ given by 
\begin{equation*}\label{Hahn 2}
Q_{n}^{\alpha,\beta,\tilde{\alpha},\tilde{\beta},N,a}(x) = \begin{pmatrix} 
H_{n}^{(\alpha,\beta,N)}(x) & a(H_{n+1}^{(\tilde{\alpha},\tilde{\beta},N)}(x)-H_{n}^{(\alpha,\beta,N)}(x)x) \\
-a\mu_{n}H_{n-1}^{(\alpha,\beta,N)}(x) & a^{2}\mu_{n}H_{n-1}^{(\alpha,\beta,N)}(x)x + H_{n}^{(\tilde{\alpha},\tilde{\beta},N)}(x)
\end{pmatrix}, \quad n=0,\dots,N,
\end{equation*}
where $\mu_{n} = \frac{(\tilde{\alpha}+1)_{n}(\tilde{\beta}+1)_{n}n(N+1-n)}{(\alpha+1)_{n-1}(\beta+1)_{n-1}(n+\alpha+\beta+N)(n+\alpha+\beta-1)}$, and $H^{(\tilde{\alpha},\tilde{\beta},N)}_{N+1}(x) = x(x-1)\cdots(x-N)$.  Applying Theorem \ref{classical bispectral}, we obtain that the sequence $Q_{n}^{\alpha,\beta,\tilde{\alpha},\tilde{\beta},N,a}(x)$ is an eigenfunction of the difference operator $D$ given by
\begin{equation*}
     \begin{split}
         D & = \Delta \begin{psmallmatrix}(x+\alpha+1)(x-N) & a(x(\tilde{\alpha}-\alpha)+\tilde{\alpha}+x+1)(x-N) \\ 0 & (x+\tilde{\alpha} + 1)(x-N) \end{psmallmatrix} + \begin{psmallmatrix} 0 & -a(N(\tilde{\alpha}+1)+x(\alpha-\tilde{\alpha}+\beta -\tilde{\beta}-2)) \\ 0 & -(\alpha + \beta) \end{psmallmatrix} \\
 & \quad - \nabla \begin{psmallmatrix} x(x-\beta-N-1) & ax(x(\beta-\tilde{\beta})+N+\tilde{\beta}-x+1) \\ 0 & x(x-\tilde{\beta}-N-1)\end{psmallmatrix}.
     \end{split}
 \end{equation*}
 It follows that 
 $$Q_{n}^{\alpha,\beta,\tilde{\alpha},\tilde{\beta},N,a}(x) \cdot D = \begin{pmatrix} n(n+\alpha+\beta+1) & 0 \\ 0 & (n-1)(n+\alpha+\beta)\end{pmatrix} Q_{n}^{\alpha,\beta,\tilde{\alpha},\tilde{\beta},N,a}(x).$$

\subsection{Matrix-valued Charlier-Meixner}
As in \cite{BP25}, where we constructed matrix-valued orthogonal polynomials by combining polynomials from different families (Hermite and Laguerre), we can follow a similar approach here by combining Charlier and Meixner polynomials.

Let $c, \beta>0$, and $0<b<1$. We consider the scalar Charlier weight $w_{c}(x) = \frac{c^{x}}{x!}$ and the scalar Meixner weight $w_{\beta,b}(x) = (\beta)_{x}\frac{b^{x}}{x!}$. For $a \neq 0$, we define the $2 \times 2$ weight matrix as in \eqref{W} 

$$
W(x) = \begin{pmatrix} \frac{c^{x}}{x!} + a^{2}x^{2}(\beta)_{x}\frac{b^{x}}{x!} & ax(\beta)_{x}\frac{b^{x}}{x!} \\ ax(\beta)_{x}\frac{b^{x}}{x!} & (\beta)_{x}\frac{b^{x}}{x!} \end{pmatrix}, \quad x\in \mathbb{N}_{0}.
$$

By Theorem \ref{main}, the corresponding sequence of orthogonal polynomials is given by  

$$
Q_{n}(x) = \begin{pmatrix} C_{n}^{(c)}(x) & a(M_{n+1}^{(\beta,b)}(x) - C_{n}^{(c)}(x)x) \\ -a \frac{(\beta)_{n}nb^{n}}{c^{n-1}(1-b)^{2n+\beta}e^{c}} C_{n-1}^{(c)}(x) & a^{2} \frac{(\beta)_{n}nb^{n}}{c^{n-1}(1-b)^{2n+\beta}e^{c}} C_{n-1}^{(c)}(x) + M_{n}^{(\beta,b)}(x)\end{pmatrix},
$$
where $C_{n}^{(c)}(x)$ and $M_{n}^{(\beta,b)}(x)$ are the $n$-th monic orthogonal polynomials associated with the weights $w_{c}(x)$ and $w_{\beta,b}(x)$, respectively.

In Theorem \ref{classical bispectral} we prove that this sequence is bispectral. We have that $Q_{n}(x)$ satisfy the second-order difference equation  

$$
Q_{n}(x) \cdot D = \begin{pmatrix} (n+1) & 0 \\ 0 & n \end{pmatrix}Q_{n}(x),
$$
where  
$$
D = \Delta \begin{psmallmatrix} -c & acx + \frac{ab}{b-1}(x+1)(x+\beta) \\ 0 & \frac{b(x+\beta)}{b-1}\end{psmallmatrix} + \begin{psmallmatrix} 1 & \frac{ab\beta}{b-1} \\ 0 & 0 \end{psmallmatrix} - \nabla \begin{psmallmatrix} x & -\frac{ax(x(b-2)+1)}{b-1} \\ 0 & \frac{x}{b-1} \end{psmallmatrix}.
$$
It also satisfies the three-term recurrence relation $Q_{n}(x)x = A_{n}Q_{n+1}(x) + B_{n}Q_{n}(x) + C_{n}Q_{n-1}(x)$, with

$$A_{n} = \begin{psmallmatrix}
1 & 
-\frac{n c a (b - 1) \left[b(n + 1)(n + \beta + 2) + c(b - 1) - n - b\right]}{m_{n}a^{2}b(n + 1)(\beta + n) + c n (b - 1)^{2}} 
\\
0 & 
\frac{c n (b - 1)^2 (m_{n}a^{2} + 1)}{m_{n}a^{2}b(n + 1)(\beta + n) + c n (b - 1)^{2}}
\end{psmallmatrix}, \quad C_{n} = \begin{psmallmatrix}
\frac{m_{n}a^{2}b(n+1)(n+\beta)}{(m_{n}a^{2} + 1)(b - 1)^{2}} + \frac{n c}{m_{n}a^{2} + 1} & 0 \\
-\frac{m_{n}a(n + c - 1)}{m_{n}a^{2} + 1} - \frac{m_{n}a b n (n + \beta)}{(m_{n}a^{2} + 1)(b - 1)} & \frac{b n (n + \beta - 1)}{(b - 1)^{2}}
\end{psmallmatrix}$$

$$B_{n} = \begin{psmallmatrix}
-\frac{b(n + 1)(n + \beta + 1)}{b - 1} + \frac{n c (b - 1)\left[b(n + 1)(\beta + 1) + b n(n + 2) + b c - c - n\right]}{m_{n}a^{2}b(n + 1)(\beta + n) + c n (b - 1)^{2}} & 
-\frac{a\left[n c (b - 1)^2 - b(n + 1)(n + \beta)\right]}{(m_{n}a^{2} + 1)(b - 1)^2} 
\\
-\frac{m_{n} a \left[n\left((b - 1)^2 c - b(n + 1)\right) - b \beta (n + 1)\right]}{m_{n}a^{2}b(n + 1)(\beta + n) + c n (b - 1)^{2}} & 
\frac{m_{n}a^{2}(n + c - 1)(b - 1) - b n(n + \beta)}{(m_{n}a^{2} + 1)(b - 1)}
\end{psmallmatrix},$$
where $m_{n} = \frac{(\beta)_{n}b^{n}n}{(1-b)^{2n+\beta}c^{n-1}e^{c}}$.

\section{Appendix}
Let $w$ be a scalar weight supported on $\{0,1,\dots,N\}$. The sequence of monic orthogonal polynomials $\{p_{n}\}_{n=0}^{N}$ satisfies a three-term recurrence relation of the form 
$$
p_{n}(x) x = p_{n+1}(x) + b_{n} p_{n}(x) + c_{n} p_{n-1}(x),
$$
where
\begin{equation*}
\begin{split}
b_{n} &= \langle p_{n}(x) x, p_{n}(x) \rangle_{w} \|p_{n}\|^{-2}, \\
c_{n} &= \langle p_{n}(x) x, p_{n-1}(x) \rangle_{w} \|p_{n-1}\|^{-2}.
\end{split}
\end{equation*}

\begin{prop}\label{N+1}
If we construct the $(N+1)$-th polynomial using the three-term recurrence relation, i.e.,
$$
p_{N+1}(x) = p_{N}(x) x - b_{N} p_{N}(x) - c_{N} p_{N-1}(x),
$$
then it follows that
$$
p_{N+1}(x) = x (x - 1) \cdots (x - N).
$$
\end{prop}

\begin{proof}
The polynomial $p_{N}(x)x$ is monic of degree $N+1$. There exist
$a_{N},\ldots,a_{0}\in\mathbb{C}$ such that the equality
$$
p_{N}(x) x = a_{N} p_{N}(x) + a_{N-1} p_{N-1}(x) + \cdots + a_{0} p_{0}(x)
$$
holds for all $x=0,\dots,N$.
Using the orthogonality of the sequence $\{p_{n}\}$, one obtains $a_{j} = 0$ for all $j < N-1$.

Therefore,
$$
p_{N}(x) x = a_{N} p_{N}(x) + a_{N-1} p_{N-1}(x),
$$
where 
$$
a_{N} = b_{N} = \langle p_{N}(x) x, p_{N}(x) \rangle_{w} \|p_{N}\|^{-2}, \quad 
a_{N-1} = c_{N} = \langle p_{N}(x) x, p_{N-1}(x) \rangle_{w} \|p_{N-1}\|^{-2}.
$$
Thus, the polynomial
$$
p_{N}(x) x - b_{N} p_{N}(x) - c_{N} p_{N-1}(x)
$$
is a monic polynomial of degree $N+1$ that vanishes at $x = 0,1,\dots,N$. Hence, it must be the polynomial
$$
x (x - 1) \cdots (x - N).
$$
\end{proof}

It is therefore natural to extend the sequence beyond the support by setting
$$
p_{N+1}(x)=x(x-1)\cdots(x-N).
$$
This is precisely the monic polynomial obtained by continuing the three-term recurrence relation to the index $N+1$. Moreover, since it vanishes on the support of $w$, it is orthogonal to every polynomial of degree at most $N$. This extension will be used in the expressions \eqref{Qn} and \eqref{Qn exp} for the matrix-valued orthogonal polynomials associated with the weight matrix $W$.

\

\bibliographystyle{amsplain} 
\bibliography{referencias} 

\begin{bibdiv}
\begin{biblist}

\bib{ADR13}{article}{
      author={{\'A}lvarez-Nodarse, R.},
      author={Dur{\'a}n, A.~J.},
      author={De~los R{\'{\i}}os, A.~M.},
       title={Orthogonal matrix polynomials satisfying second order difference equations},
        date={2013},
        ISSN={0021-9045},
     journal={J. Approx. Theory},
      volume={169},
       pages={40\ndash 55},
}

\bib{BP24-1}{misc}{
      author={Bono~Parisi, I.},
      author={Pacharoni, I.},
       title={Singular solutions of the {Matrix Bochner Problem: the $N$-dimensional cases}},
        date={2024},
        note={arXiv:2411.00798},
}

\bib{BP25}{article}{
      author={Bono~Parisi, Ignacio},
       title={Explicit construction of matrix-valued orthogonal polynomials of arbitrary size},
        date={2026},
        ISSN={1664-039X},
     journal={J. Spectr. Theory},
      volume={16},
      number={1},
       pages={299\ndash 328},
}

\bib{CY18}{article}{
      author={Casper, W.~R.},
      author={Yakimov, M.},
       title={The matrix {Bochner} problem},
        date={2022},
        ISSN={0002-9327},
     journal={Am. J. Math.},
      volume={144},
      number={4},
       pages={1009\ndash 1065},
}

\bib{CG06}{article}{
      author={Castro, M.~M.},
      author={Gr{\"u}nbaum, F.~A.},
       title={The algebra of differential operators associated to a family of matrix-valued orthogonal polynomials: five instructive examples},
        date={2006},
        ISSN={1073-7928},
     journal={Int. Math. Res. Not.},
      volume={2006},
      number={7},
       pages={33},
        note={Id/No 47602},
}

\bib{D97}{article}{
      author={Dur{\'a}n, A.~J.},
       title={Matrix inner product having a matrix symmetric second-order differential operator},
        date={1997},
        ISSN={0035-7596},
     journal={Rocky Mt. J. Math.},
      volume={27},
      number={2},
       pages={585\ndash 600},
         url={math.la.asu.edu/~rmmc/rmj/VOL27-2/CONT27-2/CONT27-2.html},
}

\bib{D12}{article}{
      author={Dur{\'a}n, A.~J.},
       title={The algebra of difference operators associated to a family of orthogonal polynomials},
        date={2012},
        ISSN={0021-9045},
     journal={J. Approx. Theory},
      volume={164},
      number={5},
       pages={586\ndash 610},
}

\bib{DG04}{article}{
      author={Dur{\'a}n, A.~J.},
      author={Gr{\"u}nbaum, F.~A.},
       title={Orthogonal matrix polynomials satisfying second-order differential equations},
        date={2004},
        ISSN={1073-7928},
     journal={Int. Math. Res. Not.},
      volume={2004},
      number={10},
       pages={461\ndash 484},
}

\bib{DG07}{article}{
      author={Dur{\'a}n, A.~J.},
      author={Gr{\"u}nbaum, F.~A.},
       title={Matrix orthogonal polynomials satisfying second-order differential equations: coping without help from group representation theory},
        date={2007},
        ISSN={0021-9045},
     journal={J. Approx. Theory},
      volume={148},
      number={1},
       pages={35\ndash 48},
}

\bib{DS14}{article}{
      author={Dur{\'a}n, A.~J.},
      author={S{\'a}nchez-Canales, V.},
       title={Rodrigues' formulas for orthogonal matrix polynomials satisfying second-order difference equations},
    language={English},
        date={2014},
        ISSN={1065-2469},
     journal={Integral Transforms Spec. Funct.},
      volume={25},
      number={11},
       pages={849\ndash 863},
}

\bib{EMR24}{article}{
      author={Eijsvoogel, B.},
      author={Morey, L.},
      author={Rom{\'a}n, P.},
       title={Duality and difference operators for matrix valued discrete polynomials on the nonnegative integers},
        date={2024},
        ISSN={0176-4276},
     journal={Constr. Approx.},
      volume={59},
      number={1},
       pages={143\ndash 227},
}

\bib{G03}{article}{
      author={Gr{\"u}nbaum, F.~A.},
       title={Matrix valued {Jacobi} polynomials},
        date={2003},
        ISSN={0007-4497},
     journal={Bull. Sci. Math.},
      volume={127},
      number={3},
       pages={207\ndash 214},
}

\bib{GPT01}{article}{
      author={Gr{\"u}nbaum, F.~A.},
      author={Pacharoni, I.},
      author={Tirao, J.},
       title={A matrix-valued solution to {Bochner}'s problem},
        date={2001},
        ISSN={0305-4470},
     journal={J. Phys. A, Math. Gen.},
      volume={34},
      number={48},
       pages={10647\ndash 10656},
}

\bib{GPT02}{article}{
      author={Gr{\"u}nbaum, F.~A.},
      author={Pacharoni, I.},
      author={Tirao, J.},
       title={Matrix valued spherical functions associated to the complex projective plane},
        date={2002},
        ISSN={0022-1236},
     journal={J. Funct. Anal.},
      volume={188},
      number={2},
       pages={350\ndash 441},
}

\bib{GPT05}{article}{
      author={Gr{\"u}nbaum, F.~A.},
      author={Pacharoni, I.},
      author={Tirao, J.},
       title={Matrix valued orthogonal polynomials of {Jacobi} type: the role of group representation theory.},
        date={2005},
        ISSN={0373-0956},
     journal={Ann. Inst. Fourier},
      volume={55},
      number={6},
       pages={2051\ndash 2068},
}

\bib{KS94}{article}{
      author={Koekoek, R.},
      author={Swarttouw, R.~F.},
       title={The askey-scheme of hypergeometric orthogonal polynomials and its q-analogue},
        date={1994},
     journal={Delft University of Technology, Faculty of Technical Mathematics and Informatics},
}

\bib{KPR12}{article}{
      author={Koelink, E.},
      author={Van~Pruijssen, M.},
      author={Rom{\'a}n, P.},
       title={Matrix-valued orthogonal polynomials related to {{\((\mathrm{SU}(2) \times \mathrm{SU}(2), \mathrm{diag})\)}}},
        date={2012},
        ISSN={1073-7928},
     journal={Int. Math. Res. Not.},
      volume={2012},
      number={24},
       pages={5673\ndash 5730},
}

\bib{K49}{article}{
      author={Krein, M.~G.},
       title={Infinite j-matrices and a matrix moment problem},
        date={1949},
        ISSN={1073-7928},
     journal={Dokl. Akad. Nauk SSSR},
      volume={69},
      number={2},
       pages={125\ndash 128},
}

\bib{K71}{article}{
      author={Krein, M.~G.},
       title={Fundamental aspects of the representation theory of {Hermitian} operators with deficiency index ({{\(m,m\)}})},
        date={1971},
        ISSN={0065-9290},
     journal={Transl., Ser. 2, Am. Math. Soc.},
      volume={97},
       pages={75\ndash 143},
}

\bib{P08}{article}{
      author={Pacharoni, I.},
       title={Matrix spherical functions and orthogonal polynomials: an instructive example},
        date={2008},
        ISSN={0041-6932},
     journal={Rev. Uni{\'o}n Mat. Argent.},
      volume={49},
      number={2},
       pages={1\ndash 15},
         url={inmabb.criba.edu.ar/revuma/revuma.php?p=toc/vol49},
}

\bib{PR08}{article}{
      author={Pacharoni, I.},
      author={Roman, P.},
       title={A sequence of matrix valued orthogonal polynomials associated to spherical functions},
        date={2008},
        ISSN={0176-4276},
     journal={Constr. Approx.},
      volume={28},
      number={2},
       pages={127\ndash 147},
}

\bib{PT07}{article}{
      author={Pacharoni, I.},
      author={Tirao, J.},
       title={Matrix valued orthogonal polynomials arising from the complex projective space},
        date={2007},
        ISSN={0176-4276},
     journal={Constr. Approx.},
      volume={25},
      number={2},
       pages={177\ndash 192},
}

\bib{PZ16}{article}{
      author={Pacharoni, I.},
      author={Zurri{\'a}n, I.},
       title={Matrix {Gegenbauer} polynomials: the {{\(2\times 2\)}} fundamental cases},
        date={2016},
        ISSN={0176-4276},
     journal={Constr. Approx.},
      volume={43},
      number={2},
       pages={253\ndash 271},
}

\end{biblist}
\end{bibdiv}

\end{document}